\documentclass[reqno,11pt]{amsart}
\usepackage{eurosym}
\usepackage{amssymb}
\usepackage{tikz}
\usepackage{amsmath}
\usepackage{tikz}
\usepackage{pgflibraryplotmarks}
\usepackage{wasysym}

\usepackage{pgfplots} 
\usetikzlibrary{patterns}

\makeatletter
\newcommand{\pgfplotsdrawaxis}{\pgfplots@draw@axis}
\makeatother
\pgfplotsset{only axis on top/.style={axis on top=false, after end axis/.code={
			\pgfplotsset{axis line style=opaque, ticklabel style=opaque, tick style=opaque,
				grid=none}\pgfplotsdrawaxis}}}

\newcommand{\drawge}{-- (rel axis cs:1,0) -- (rel axis cs:1,1) -- (rel axis cs:0,1) \closedcycle}

\theoremstyle{plain}
\newtheorem{theorem}{Theorem}[section]
\newtheorem{proposition}[theorem]{Proposition}
\newtheorem{lemma}[theorem]{Lemma}
\newtheorem{corollary}[theorem]{Corollary}

\theoremstyle{definition}
\newtheorem{definition}[theorem]{Definition}
\newtheorem{example}[theorem]{Example}
\newtheorem{remark}[theorem]{Remark}
\newtheorem{notation}[theorem]{Notation}

\DeclareMathOperator*{\msup}{msup}

\begin{document}
\def\sect#1{\section*{\leftline{\large\bf #1}}}
\def\th#1{\noindent{\bf #1}\bgroup\it}
\def\endth{\egroup\par}
\def\minf{\textnormal{minf}}
\def\spn{\textnormal{span}}
\def\N{\mathbb{N}}
\def\R{\mathbb{R}}
\def\Q{\mathbb{Q}}
\def\W{\mathcal{W}}
\def\V{\mathcal{V}}
\def\L{\mathcal{L}}
\def\D{\mathcal{D}}

\title[Riesz-Kantorovich Formulas for Operators on Multi-Wedged Spaces]{Riesz-Kantorovich Formulas for Operators on Multi-Wedged Spaces}

\author[C. Schwanke]{Christopher Schwanke}
\email{schwankc326@gmail.com}
\address{Unit for BMI\\North-West University\\Private Bag X6001\\Potchefstroom 2520\\South Africa}

\author[M. Wortel]{Marten Wortel}
\email{marten.wortel@gmail.com}
\address{Unit for BMI\\North-West University\\Private Bag X6001\\Potchefstroom 2520\\South Africa}



\begin{abstract}
We introduce the notions of multi-suprema and multi-infima for vector spaces equipped with a collection of wedges, generalizing the notions of suprema and infima in ordered vector spaces. Multi-lattices are vector spaces closed under multi-suprema and multi-infima and are thus an abstraction of vector lattices. The Riesz decomposition property in the multi-wedged setting is also introduced, leading to Riesz-Kantorovich formulas for multi-suprema and multi-infima in certain spaces of operators.
\end{abstract}

\maketitle

\section{Introduction}\label{S:intro}

The significance of the Riesz-Kantorovich formulas is well known not only in duality theory in ordered vector spaces but also in equilibrium theory in economics (see \cite{ATY2,ATY}). Surely, such formulas will also be needed in the emerging theory of multi-wedged spaces, i.e., vector spaces equipped with a collection of wedges. The goal of this paper is to prove these formulas in this more general setting.

Two recent papers \cite{JeuMess,Mess} have investigated Banach spaces equipped with an arbitrary collection of closed wedges. Indeed, \cite[Theorem 4.1]{JeuMess} generalizes And\^o's Theorem \cite[Lemma~1]{Ando} to such spaces, and \cite{Mess} extends geometric duality theory of real preordered Banach spaces to multi-wedged Banach spaces. Since the Riesz-Kantorovich formulas do not involve any topology, we consider multi-wedged spaces not equipped with any topological structure.

Required for the construction of multi-wedged Riesz-Kantorovich formulas is the generalization of the notion of suprema in ordered vector spaces to multi-wedged spaces. In this prospect, recall that for an ordered vector space $(E,K)$ and a collection $(x_{i})_{i\in I}$ in $E$ it is true that $z=\underset{i\in I}{\sup}\{x_{i}\}$ if and only if $\bigcap_{i\in I}(x_{i}+K)=z+K$. (see \cite[Theorem 1.16]{AT}). More generally, if $E$ is a vector space, $(W_{i})_{i\in I}$ is a collection of wedges in $E$, and $(x_{i})_{i\in I}$ is a collection of elements of $E$ then any $z\in E$ that satisfies
\[
\bigcap_{i\in I}(x_{i}+W_{i})=z+\bigcap_{i\in I}W_{i}
\]
can be viewed as a generalized supremum, which we call a multi-supremum, of $(x_{i},W_{i})_{i\in I}$. In order for such a multi-supremum to exist, $(x_{i},W_{i})_{i\in I}$ must be multi-bounded above, meaning that $\bigcap_{i\in I}(x_{i}+W_{i})\neq\varnothing$. Multi-wedged spaces in which all multi-bounded above collections $(x_{i},W_{i})_{i\in I}$ with $|I|<\kappa+1$ are closed under multi-suprema are called $\kappa$-multi-lattices.

The venturing from vector lattices to the more general multi-lattices requires a few compromises, some of which come as surprises.
\begin{itemize}
	\item[(1)] Multi-suprema are generally not unique [Remark~\ref{R:setofmsup}].
	\item[(2)] An $n$-multi-lattice need not be an $(n+1)$-multi-lattice [Example~\ref{E:not_n+1}].
	\item[(3)] An $n$-multi-lattice need not have the Riesz decomposition property [Example~\ref{E:noRDP}].
\end{itemize}

As alluded to above, the Riesz decomposition property (as well as the concept of Dedekind completeness) is expanded to the multi-wedged setting in a manner suitable for obtaining Riesz-Kantorovich formulas for multi-bounded above collections of operators between a multi-wedged space with this generalized Riesz decomposition property and a Dedekind complete multi-lattice with a single wedge [Theorem~\ref{T:main}]. Introspection of the proof of Theorem~\ref{T:main} reveals that the space of order bounded operators $\L_{b}(E,F)$ is a Dedekind complete vector lattice when $E$ is a preordered vector space with generating wedge and the Riesz decomposition property and $F$ is a Dedekind complete multi-lattice equipped with a single cone [Corollary~\ref{C:classic}].

Both authors would like to acknowledge that the idea of multi-lattices and the possibility of multi-lattices being in duality with a generalized Riesz decomposition property were first concocted by Rabee Tourky in a conversation with Miek Messerschmidt at Australian National University in April 2014. We are grateful to Miek for suggesting the topic to us.

\section{Multi-Wedged Spaces and Multi-Lattices}\label{S:prelims}

We refer the reader to \cite{AT} for any unexplained terminology and basic results regarding wedges, cones, preordered and ordered vector spaces, and vector lattices. All vector spaces in this paper are over $\R$.

For a vector space $E$, we call a nonempty subset $W$ of $E$ a \textit{wedge} if $W+W\subseteq W$ and $\lambda W\subseteq W$ for all $\lambda\geq 0$. A wedge $K$ that satisfies $K\cap(-K)=\{0\}$ is called a \textit{cone}. Any wedge $W$ can be written as the sum of the vector space $W\cap(-W)$ and a cone. For brevity, we denote $W\cap(-W)$ by $\mathcal{D}(W)$.

Given an indexed collection $(W_{i})_{i\in I}$ of wedges in a vector space $E$, we define
\[
\sum_{i\in I}W_{i}=\left\{\sum_{k=1}^{n}x_{k}:x_{k}\in\bigcup_{i\in I}W_{i}, n\in\N\right\}.
\]
Thus any element $\sum_{i\in I}x_{i}\in\sum_{i\in I}W_{i}$ is a finite sum. We note that $\sum_{i\in I}W_{i}$ is the smallest wedge in $E$ that contains each $W_{i}$.

\begin{definition}\label{D:mws}
We call a pair $(E,\W)$ a \textit{multi-wedged space} if $E$ is a vector space and $\W$ is a nonempty collection of wedges in $E$.
\end{definition}

\begin{notation}
Let $(E,\W)$ be a multi-wedged space. Each $W\in\W$ induces a preordering on $E$. For $x,y\in E$, we write $x\leq_{W}y$ (or $y\geq_{W}x$) whenever $y-x\in W$. In addition, we write $x\sim_{W}y$ when $x\leq_{W}y$ and $y\leq_{W}x$, or equivalently, when $x-y\in\mathcal{D}(W)$. Given $A,B\subseteq E$, we write $A\leq_{W}B$ (respectively, $A\sim_{W}B$) if $a\leq_{W}b$ (respectively, $a\sim_{W}b$) for each $a\in A$ and $b\in B$. For $x\in E$ and $A\subseteq E$, the notation $x\leq_{W}A$ (respectively, $x\sim_{W}A$) means $\{x\}\leq_{W}A$ (respectively, $\{x\}\sim_{W}A$). Given an indexed collection $(W_{i})_{i\in I}$ in $\W$ and $x,y\in E$, we write $x\leq_{i}y$ instead of $x\leq_{W_{i}}y$.
\end{notation}

\begin{definition}\label{D:msup}
Let $(E,\W)$ be a multi-wedged space. A collection $(x_{i},W_{i})_{i\in I}$ in $E\times\W$ is called \textit{multi-bounded above} if there exists $u\in E$ such that $x_{i}\leq_{i}u$ for each $i\in I$. In this case, $u$ is called a \textit{multi-upper bound} of $(x_{i},W_{i})_{i\in I}$. The definitions of \textit{multi-bounded below} and \textit{multi-lower bound} are defined analogously. Furthermore, we say $z\in E$ is a \textit{multi-supremum} of $(x_{i},W_{i})_{i\in I}$ whenever
\begin{itemize}
\item[(1)] $z$ is a multi-upper bound of $(x_{i},W_{i})_{i\in I}$,
\item[(2)] $u\geq_{i} x_{i}$ for every $i\in I$ implies $u\geq_{i}z$ for each $i\in I$.
\end{itemize}
\textit{Multi-lower bounds} and \textit{multi-infima} are defined similarly.
\end{definition}

We denote the set of all multi-suprema, respectively multi-infima of $(x_{i},W_{i})_{i\in I}$ by
\[
\underset{i\in I}{\msup}(x_{i},W_{i}),\quad \text{respectively}\quad \underset{i\in I}{\minf}(x_{i},W_{i}).
\]
In the case that $W_{i}=W\ (i\in I)$ for some $W\in\W$, we will instead write $\underset{i\in I}{\msup}\{x_{i}\}$, or when convenient, $\msup\{x_{i}:i\in I\}$. In this case, it will always be clear from the context which wedge is used in the set of multi-suprema. The special case where $\underset{i\in I}{\msup}(x_{i},W_{i})=\{z\}$ for some $z\in E$ will naturally be of special interest. In this scenario, we say that the multi-supremum is \textit{proper} and simply write $\underset{i\in I}{\msup}(x_{i},W_{i})=z$.

\begin{remark}\label{R:setofmsup}
Given a multi-wedged space $(E,\W)$ and a collection $(x_{i},W_{i})_{i\in I}$ in $E\times\W$, it easily follows from the definition of multi-suprema that for any $z\in\underset{i\in I}{\msup}(x_{i},W_{i})$,
\begin{align*}
\underset{i\in I}{\msup}(x_{i},W_{i})&=\{w:w\sim_{i}z\ \text{for all}\ i\in I\}\\
&=z+\D\left(\bigcap_{i\in I}W_{i}\right).
\end{align*}
Hence one multi-supremum determines the entire set of multi-suprema. We also see that a nonempty $\underset{i\in I}{\msup}(x_{i},W_{i})$ is proper if and only if $\bigcap_{i\in I}W_{i}$ is a cone.
\end{remark}

Multi-suprema and multi-infima in multi-wedged spaces share some familiar properties with classical suprema and infima in ordered vector spaces. We record these useful properties in the following proposition, whose straightforward proof is left to the reader.

\begin{proposition}\label{P:ids}
	Let $(E,\W)$ be a multi-wedged space, and let $(x_{i},W_{i})_{i\in I}$ be a collection in $E\times\W$. The following hold.
	\begin{itemize}
		\item[(1)] $\underset{i\in I}{\minf}(x_{i},W_{i})=-\underset{i\in I}{\msup}(-x_{i},W_{i})$.
		\item[(2)] For every $y\in E$ we have $y+\underset{i\in I}{\msup}(x_{i},W_{i})=\underset{i\in I}{\msup}(y+x_{i},W_{i})$.
		\item[(3)] For every $\lambda>0$ we have $\lambda\,\underset{i\in I}{\msup}(x_{i},W_{i})=\underset{i\in I}{\msup}(\lambda x_{i},W_{i})$.
	\end{itemize}
\end{proposition}

As hinted at in parts (2) and (3)  above, we mostly focus on multi-suprema in this paper, since (1) of the previous proposition can be used to transfer the results of this paper to analogous results for multi-infima.

For an ordered vector space $(E,K)$ and a collection $(x_{i})_{i\in I}$ in $E$, there is a convenient geometrical interpretation of $\underset{i\in I}{\sup}\{x_{i}\}$ (as mentioned in the introduction). Indeed, $\underset{i\in I}{\sup}\{x_{i}\}$ exists in $E$ if and only if there exists $z\in E$ such that $\bigcap_{i\in I}(x_{i}+K)=z+K$, and in this case, $z=\underset{i\in I}{\sup}\{x_{i}\}$ (see \cite[Theorem 1.16]{AT}). The next proposition shows that there exists a useful geometric interpretation for multi-suprema as well. This geometric interpretation will prove convenient for determining whether or not a multi-wedged space is a multi-lattice (see Example~\ref{E:not_n+1}).

\begin{proposition}\label{P:geoint} Let $(E,\W)$ be a multi-wedged space, and let $(x_{i},W_{i})_{i\in I}$ be a collection in $E\times\W$. For $z\in E$, we have $z\in\underset{i\in I}{\msup}(x_{i},W_{i})$ if and only if
\[
\bigcap_{i\in I}(x_{i}+W_{i})=z+\bigcap_{i\in I}W_{i}.
\]
\end{proposition}

\begin{proof}
First assume that $z\in\underset{i\in I}{\msup}(x_{i},W_{i})$. Condition (1) in Definition~\ref{D:msup} implies that for any given $i\in I$ we have $z-x_{i}\in W_{i}$ and thus $z+W_{i}\subseteq x_{i}+W_{i}$. Hence
\[
z+\bigcap_{i\in I}W_{i}=\bigcap_{i\in I}(z+W_{i})\subseteq\bigcap_{i\in I}(x_{i}+W_{i}).
\]
On the other hand, suppose that $u\in\bigcap_{i\in I}(x_{i}+W_{i})$. Then we have $u\geq_{i}x_{i}$ for each $i\in I$. It follows from condition (2) of Definition~\ref{D:msup} that $u\geq_{i}z$ for each $i\in I$. We thus have $u\in z+\bigcap_{i\in I}W_{i}$. Hence
\[
\bigcap_{i\in I}(x_{i}+W_{i})=z+\bigcap_{i\in I}W_{i}.
\]
	
Conversely, suppose that $\bigcap_{i\in I}(x_{i}+W_{i})=z+\bigcap_{i\in I}W_{i}$. Let $i\in I$. Then
\[
z\in z+\bigcap_{i\in I}W_{i}=\bigcap_{i\in I}(x_{i}+W_{i})\subseteq x_{i}+W_{i},
\]
and thus $z\geq_{i}x_{i}$. Suppose next that $u\geq_{i}x_{i}$ for each $i\in I$. It follows that
\[
u\in\bigcap_{i\in I}(x_{i}+W_{i})=z+\bigcap_{i\in I}W_{i},
\]
which implies $u\geq_{i}z$ for all $i\in I$. Therefore, $z\in\underset{i\in I}{\msup}(x_{i},W_{i})$.
\end{proof}

\begin{definition}\label{D:ml} (1) For a cardinal number $\kappa$, we call a multi-wedged space $(E,\W)$ a $\kappa$-\textit{multi-lattice} if for every multi-bounded above collection $(x_{i},W_{i})_{i\in I}$ in $E\times\W$ with $|I|<\kappa+1$ we have $\underset{i\in I}{\msup}(x_{i},W_{i})\neq\varnothing$. 
\begin{itemize}
\item[(2)] A \textit{Dedekind complete multi-lattice} is a multi-wedged space that is a $\kappa$-multi-lattice for every cardinal number $\kappa$.
\item[(3)] For brevity, we refer to a $2$-multi-lattice as a \textit{multi-lattice}.
\end{itemize}
In particular, an \textit{$\aleph_{0}$-multi-lattice} refers to a multi-wedged space that is an $n$-multi-lattice for every $n\in\N$.
\end{definition}

\begin{remark}\label{R:induction}
A standard induction argument proves the following. If $(E,\W)$ is a multi-lattice and $\W$ is closed under finite intersections then $(E,\W)$ is an $\aleph_{0}$-multi-lattice.
\end{remark}

Given $n\in\N$, it is not true however, that every $n$-multi-lattice is an $(n+1)$-multi-lattice, as the following example reveals.

\begin{example}\label{E:not_n+1}
Let $E=\R^{2}$, let $W_{1}=\{(x,y)\in\R^{2}:x\geq 0\}$, $W_{2}=\{(x,y)\in\R^{2}:y\geq 0\}$, and $W_{3}=\{(x,y)\in\R^{2}:y\geq-x\}$. Then $(E,\{ W_{1},W_{2},W_{3}\})$ is a multi-lattice that is not a $3$-multi-lattice.
\end{example}

Indeed, it is readily checked using Proposition~\ref{P:geoint} that $(E,\{ W_{1},W_{2},W_{3}\})$ is a multi-lattice. However, $W_{1}\cap W_{2}\cap\big((1,1)+W_{3}\big)$, which constitutes the shaded region in Figure 1 below, is not a translated wedge. Another appeal to Proposition~\ref{P:geoint} verifies our claim.

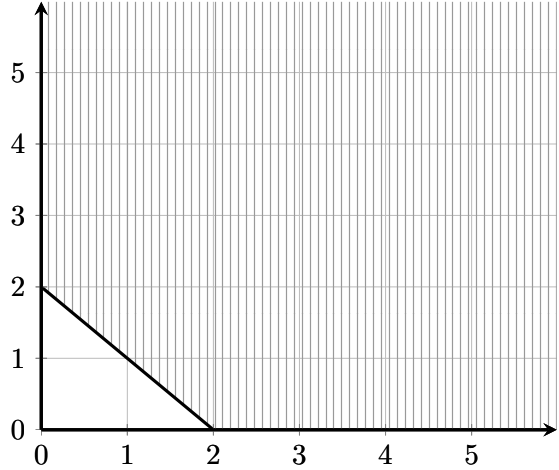
\begin{figure}[h!] 
	\centering
	\label{fig:p3:c1} 
	\begin{tikzpicture} 
	\begin{axis}[only axis on top,
	axis line style=very thick, 
	axis x line=bottom, 
	axis y line=left, 
	ymin=0,ymax=5.99,xmin=0,xmax=5.99, 
grid=major 
	] 
	\addplot [draw=none, pattern=vertical lines, pattern color=black!40, domain=-10:4]
	{-x+2} \drawge; 
	\addplot[very thick, domain=-10:10] {-x+2}; 
	\end{axis} 
	\end{tikzpicture} 
	\caption{$W_{1}\cap W_{2}\cap\left((1,1)+W_{3}\right)$} 
\end{figure} 

\section{Riesz-Kantorovich Formulas}\label{S:ops}

Operators between multi-wedged spaces and multi-lattices are investigated in this section, obtaining Riesz-Kantorovich formulas in the multi-wedged setting [Theorem~\ref{T:main}]. We begin with a notion of positive operators between multi-wedged spaces.

\begin{definition}\label{D:pos}
Let $(E,\W)$ and $(F,\V)$ be multi-wedged spaces. For $W\in\W $ and $V\in\V$, we say that a map $T\colon E\rightarrow F$ is $(W,V)$-\textit{positive} if $T(W)\subseteq V$. We denote by $\L_{W,V}(E,F)$ the collection of all $(W,V)$-positive operators $T\colon E\rightarrow F$. Also, we set
\[
\L_{\W,\V}(E,F)=\{\L_{W,V}(E,F):W\in\W, V\in\V\}.
\]
In the case that $\V=\{V\}$ for some wedge $V$ in $F$, we write $\L_{\W,V}(E,F)$ instead of $\L_{\W,\{V\}}(E,F)$, for short.
\end{definition}

The following proposition is needed for our main result, Theorem~\ref{T:main}. We call a wedge $W$ in a vector space $E$ \textit{generating} if $E=W-W$.

\begin{proposition}\label{P:wedgeofops}
If $(E,\W)$ and $(F,\V)$ are multi-wedged spaces then $\L_{W,V}(E,F)$ is a wedge in $\L(E,F)$ for each $W\in\W$ and every $V\in\V$. Thus $\L_{\W,\V}(E,F)$ is a collection of wedges in $\L(E,F)$. Moreover, given collections $(W_{i})_{i\in I}$ and $(V_{j})_{j\in J}$ in $\W$ and $\V$, respectively,
\[
\mathcal{D}\left(\bigcap_{i\in I,j\in J}\mathcal{L}_{W_{i},V_{j}}(E,F)\right)=\left\{T\in\mathcal{L}(E,F):T\left(\sum_{i\in I}W_{i}\right)\subseteq\mathcal{D}\left(\bigcap_{j\in J}V_{j}\right)\right\}.\tag{$\ast$}
\]
In particular, for $E,F\neq\{0\}$ we have that $\bigcap_{i\in I,j\in J}\mathcal{L}_{W_{i},V_{j}}(E,F)$ is a cone if and only if $\sum_{i\in I}W_{i}$ is generating and $\bigcap_{j\in J}V_{j}$ is a cone.
\end{proposition}

\begin{proof}
Let $(E,\W)$ and $(F,\V)$ be multi-wedged spaces. It is readily checked that $\L_{W,V}(E,F)$ is a wedge in $\L(E,F)$ for every $W\in\W$ and all $V\in\V$. Let $(W_{i})_{i\in I}$ and $(V_{j})_{j\in J}$ be collections in $\W$ and $\V$, respectively. Assume $T\in\mathcal{D}\left(\bigcap_{i\in I,j\in J}\mathcal{L}_{W_{i},V_{j}}(E,F)\right)$, and let $i\in I$. Suppose $x\in W_{i}$. Since $\pm T\in\mathcal{L}_{W_{i},V_{j}}(E,F)$ for every $j\in J$, we have $T(x)\in\mathcal{D}\left(\bigcap_{j\in J}V_{j}\right)$. Since $i\in I$ was chosen arbitrarily and $T$ is linear, we have $T(\sum_{i\in I}W_{i})\subseteq\mathcal{D}\left(\bigcap_{j\in J}V_{j}\right)$.

On the other hand, assume that $T\in\mathcal{L}(E,F)$ satisfies $T(\sum_{i\in I}W_{i})\subseteq\mathcal{D}\left(\bigcap_{j\in J}V_{j}\right)$. Let $i\in I,j\in J$, and let $x\in W_{i}$. By assumption, $\pm T(x)\in\bigcap_{j\in J}V_{j}\subseteq V_{j}$. Thus we have $\pm T\in\L_{W_{i},V_{j}}(E,F)$ for every $i\in I,j\in J$. Hence $T\in\mathcal{D}\left(\bigcap_{i\in I,j\in J}\mathcal{L}_{W_{i},V_{j}}(E,F)\right)$. This proves $(\ast)$.

Next assume $\sum_{i\in I}W_{i}$ is generating and $\bigcap_{j\in J}V_{j}$ is a cone. Then the only $T\in\mathcal{L}(E,F)$ such that $T\left(\sum_{i\in I}W_{i}\right)\subseteq\mathcal{D}(\bigcap_{j\in J}V_{j})$ is the zero map. Therefore, from $(\ast)$, we have that $\bigcap_{i\in I,j\in J}\mathcal{L}_{W_{i},V_{j}}(E,F)$ is a cone.

Suppose $E,F\neq\{0\}$ and that $\sum_{i\in I}W_{i}-\sum_{i\in I}W_{i}\neq E$. Let $U$ be an algebraic complement of $\sum_{i\in I}W_{i}-\sum_{i\in I}W_{i}$, and let $S\colon U\rightarrow F$ be a nonzero linear map. Next define $T\colon\left(\sum_{i\in I}W_{i}-\sum_{i\in I}W_{i}\right)\oplus U\rightarrow F$ by $T(w\oplus u)=S(u)$. Then $T$ is a nonzero linear map satisfying $T(\sum_{i\in I}W_{i})\subseteq\mathcal{D}(\bigcap_{j\in J}V_{j})$. By $(\ast)$, $\bigcap_{i\in I,j\in J}\mathcal{L}_{W_{i},V_{j}}(E,F)$ is not a cone. Finally, assume $\bigcap_{j\in J}V_{j}$ is not a cone. Let $y\in\mathcal{D}\left(\bigcap_{j\in J}V_{j}\right)$ be nonzero. Let $\phi$ be a nonzero linear functional on $E$ and define $T(x)=\phi(x)y\ (x\in E)$. Let $i\in I$ and $j\in J$ be arbitrary, and let $x\in W_{i}$. Then $\pm T(x)=\pm\phi(x)y\in\bigcap_{j\in J}V_{j}\subseteq V_{j}$. It follows that $\pm T$ are nonzero members of $\bigcap_{i\in I,j\in J}\mathcal{L}_{W_{i},V_{j}}(E,F)$.
\end{proof}

As a result of the proposition above, we obtain the following corollary regarding multi-suprema.

\begin{corollary}\label{C:pmsupofops}
Let $(E,\W)$ and $(F,\V)$ be multi-wedged spaces with $E,F\neq\{0\}$, and let $\bigl(T_{i},\L_{W_{i},V_{i}}(E,F)\bigr)_{i\in I}$ be a collection in $\L(E,F)\times\L_{\W,\V}(E,F)$ satisfying
\[
\underset{i\in I}{\msup}\bigl(T_{i},\L_{W_{i},V_{i}}(E,F)\bigr)\neq\varnothing.
\]
We have that $\underset{i\in I}{\msup}\bigl(T_{i},\L_{W_{i},V_{i}}(E,F)\bigr)$ is proper if and only if $\sum_{i\in I}W_{i}$ is generating and $\bigcap_{i\in I}V_{i}$ is a cone.
\end{corollary}

The following lemma and proof are analogous to the Riesz-Kantorovich Extension Theorem for positive additive maps on a cone (see e.g.\ \cite[Lemma 1.26]{AT}) but involve wedges instead of cones, make no reference to positive maps, require the maps in question to be positively homogeneous, and hold for general vector spaces in the codomain.

\begin{lemma}\label{L:extn}
Suppose that $(E,W)$ is a pre-ordered vector space and that $V$ is a vector space. If $T\colon W\rightarrow V$ is an additive, positively homogeneous mapping then $T$ extends to an operator $T\colon E\rightarrow V$. Moreover, if $W$ is generating then $T$ extends uniquely to all of $E$.
\end{lemma}

\begin{proof}
Let $T\colon W\rightarrow V$ be an additive, positively homogeneous mapping. Set $X=W-W$. For each $x\in X$, let $x_{1},x_{2}\in W$ be such that $x=x_{1}-x_{2}$, and define
\[
S(x)=T(x_{1})-T(x_{2}).
\]
Following the proof of \cite[Lemma 1.26]{AT}, one readily proves that $S$ is a well-defined, additive map on $X$ that extends $T$. It follows that $S$ is $\Q$-linear on $X$. We show that $S$ is also homogeneous on $X$. For this task, let $x=x_{1}-x_{2}$, with $x_{1},x_{2}\in W$, and let $\lambda\in\R$. If $\lambda\geq 0$ then $\lambda x_{1},\lambda x_{2}\in W$ and
\[
S(\lambda x)=S(\lambda x_{1}-\lambda x_{2})=T(\lambda x_{1})-T(\lambda x_{2})=\lambda T(x_{1})-\lambda T(x_{2})=\lambda S(x).
\]
Hence $S$ is positively homogeneous. For $\lambda<0$, it follows from the $\Q$-linearity and positive homogeneity of $S$ that
\[
S(\lambda x)=-S(-\lambda x)=-(-\lambda)S(x)=\lambda S(x).
\]
Thus $S$ is linear on $X$. Next let $Y$ be an algebraic complement of $X$ in $E$. Let $L\colon Y\rightarrow V$ be an operator, and define $R\colon E\rightarrow V$ by $R(x\oplus y)=S(x)+L(y)\ (x\in X,y\in Y)$. Then $R$ is a linear map on $E$ that extends $T$.

Finally, it is readily checked that $T$ extends uniquely to $E$ in the case that $W$ is generating.
\end{proof}

As well known, the Riesz decomposition property is employed in the proof of the classical Riesz-Kantorovich formulas \cite[Theorem~1.59]{AT}. Likewise, we will make use of a suitable Riesz decomposition property for multi-wedged spaces.

\begin{definition}\label{D:RDP}
Given $\alpha,\beta\in\N\cup\{\aleph_{0}\}$, we say that a multi-wedged space $(E,\W)$ has the $(\alpha,\beta)$-\textit{Riesz decomposition property} if for any $m<\alpha+1,n<\beta+1$, any collection $(W_{j})_{j\in J}$ of wedges in $\W$ with $|J|=n$, and any two collections $(x_{i})_{i\in I}$ and $(y_{j})_{j\in J}$ of vectors in $E$ such that $|I|=m$, $x_{i}\in\sum_{j\in J}W_{j}\ (i\in I)$, $y_{j}\in W_{j}\ (j\in J)$, and
\[
\sum_{i\in I}x_{i}=\sum_{j\in J}y_{i},
\]
there exist $z_{ij}\in W_{j}\ (i\in I,j\in J)$ such that
\[
x_{i}=\sum\limits_{j\in J}z_{ij}\ (i\in I)\ \text{and}\ y_{j}=\sum\limits_{i\in I}z_{ij}\ (j\in J).
\]
\end{definition}

It is known that every vector lattice has the classical Riesz decomposition property (see e.g.\ \cite[Theorem 1.20]{AB}), which coincides with the $(\aleph_{0},\aleph_{0})$-Riesz decomposition property on a vector lattice \cite[Theorem~1.54]{AT}. However, not every multi-lattice exhibits even the $(2,2)$-Riesz decomposition property, as the following simple example illustrates.

\begin{example}\label{E:noRDP}
Let $(E,\W)$ be the Dedekind complete multi-lattice with $E=\R^{2}$ and $\W=\{W_{1},W_{2}\}$, where $W_{1}=\{(x,y)\in\R^{2}:x,y\geq 0\}$ and $W_{2}=\{(x,x)\in\R^{2}:x\geq 0\}$. Then $(E,\{W_{1},W_{2}\})$ does not possess the $(2,2)$-Riesz decomposition property.
\end{example}

Indeed, it is readily checked using Proposition~\ref{P:geoint} that $(E,\{W_{1},W_{2}\})$ is a Dedekind complete multi-lattice. In order to verify that $(E,\{W_{1},W_{2}\})$ does not have the $(2,2)$-Riesz decomposition property, let $y_{1}=(1,0)$ and $y_{2}=(1,1)$. Then $y_{1}\in W_{1},y_{2}\in W_{2}$, and $y_{1}+y_{2}=(2,1)$. Next let $x_{1}=(2,0)$ and $x_{2}=(0,1)$. Then $x_{1},x_{2}\in W_{1}+W_{2}=W_{1}$ and $x_{1}+x_{2}=(2,1)$. Let $(z_{11},z_{11}'), (z_{21},z_{21}')\in W_{1}$ and $(z_{12},z_{12}), (z_{22},z_{22})\in W_{2}$, so that $z_{11},z_{11}',z_{21},z_{21}',z_{12},z_{22}\geq 0$. Suppose that $(z_{11},z_{11}'), (z_{21},z_{21}'),(z_{12},z_{12})$, and $(z_{22},z_{22})$ satisfy the following:
\begin{itemize}
\item[(1)] $(z_{11},z'_{11})+(z_{21},z_{21}')=(1,0)$,
\item[(2)] $(z_{12},z_{12})+(z_{22},z_{22})=(1,1)$,
\item[(3)] $(z_{11},z_{11}')+(z_{12},z_{12})=(2,0)$,
\item[(4)] $(z_{21},z_{21}')+(z_{22},z_{22})=(0,1)$.
\end{itemize}
From (1) we have $z_{11}'=z_{21}'=0$. That $z_{12}=0$ now follows from (3). From (3) again, we get $z_{11}=2$, and hence, from (1), we have $z_{21}=-1$, a contradiction.

For contrast, Example~\ref{E:F(S)} furnishes examples of multi-lattices with the $(\aleph_{0},\aleph_{0})$-Riesz decomposition property.

With the Riesz decomposition property for multi-wedged spaces introduced, we are almost ready to present the main result of this paper. We need the following lemma.

\begin{lemma}\label{L:P}
Let $(F,V)$ be a preordered vector space, and let $P_{\D(V)}$ be a projection onto $\D(V)$, so that $P_{U}=I-P_{\D(V)}$ is a projection onto an algebraic complement $U$ of $\D(V)$.
\begin{itemize}
\item[(1)] For $x\in E$ we have $x\sim_{V}P_{U}(x)$.
\item[(2)] If $A,B\subseteq F$ are nonempty and $A\sim_{V}B$ then $P_{U}(A)$ and $P_{U}(B)$ are singleton sets and also $P_{U}(A)=P_{U}(B)$.
\end{itemize}
\end{lemma}

\begin{proof}
(1) Let $x\in E$. Then $x-P_{U}(x)=(I-P_{U})(x)=P_{\D(V)}(x)\in\D(V)$.\\
(2) Let $A,B\subseteq F$ be nonempty sets such that $A\sim_{V}B$. Let $a\in A,b\in B$. Then $a\sim_{V}b$, so $a-b\in\D(V)$. Thus $P_{U}(a-b)=0$, and so $P_{U}(a)=P_{U}(b)$. This proves that $P_{U}(A)=P_{U}(B)$. Finally, note that if $a,a'\in A$ then for all $b\in B$ we have $P_{U}(a)=P_{U}(b)=P_{U}(a')$. Therefore, $P_{U}(A)$ is a singleton set.
\end{proof}

In Theorem~\ref{T:main} below, we consider singleton subsets of $F$ to be elements of $F$. 

\begin{theorem}\label{T:main}
Let $(E,\W)$ be a multi-wedged space, and let $(F,\{V\})$ be a Dedekind complete multi-lattice. Let $P_{\D(V)}$ be a projection onto $\D(V)$, so that $P_{U}=I-P_{\D(V)}$ is a projection onto an algebraic complement $U$ of $\D(V)$. Let $\kappa\in\N\cup\{\aleph_{0}\}$, and suppose $(E,\W)$ has the $(2,\kappa)$-Riesz decomposition property. The following hold.
\begin{itemize}
\item[(1)] $\bigl(\L(E,F),\L_{\W,V}(E,F)\bigr)$ is a $\kappa$-multi-lattice.
\item[(2)] If $\kappa=\aleph_{0}$ then $\big(\L(E,F),\L_{\W,V}(E,F)\big)$ is Dedekind complete.
\item[(3)] Let $\bigl(T_{i},\L_{W_{i},V}(E,F)\bigr)_{i\in I}$ in $\L(E,F)\times\L_{\W,V}(E,F)$ be multi-bounded above and suppose $|I|<\kappa+1$. For $R\in\L(E,F)$, we have $R\in\underset{i\in I}{\msup}\big(T_{i},\L_{W_{i},V}(E,F)\big)$ if and only if there exists $L\in\L(E,F)$ such that $L(\sum_{i\in I}W_{i})\subseteq\D(V)$ and
\[
R(x)=P_{U}\left(\msup\left\{\sum_{i\in I}T_{i}(y_{i}):y_{i}\in W_{i},\sum_{i\in I}y_{i}=x\right\}\right)+L(x)\quad \left(x\in\sum_{i\in I}W_{i}\right).
\]
In particular, if $\sum_{i\in I}W_{i}$ is generating and $V$ is a cone then $\underset{i\in I}{\msup}\big(T_{i},\L_{W_{i},V}(E,F)\big)$ is proper and
\[
\underset{i\in I}{\msup}\big(T_{i},\L_{W_{i},V}(E,F)\big)(x)=\sup\left\{\sum_{i\in I}T_{i}(y_{i}):y_{i}\in W_{i},\sum_{i\in I}y_{i}=x\right\}\quad \left(x\in\sum_{i\in I}W_{i}\right).
\]
\item[(4)] Let $\bigl(T_{i},\L_{W_{i},V}(E,F)\bigr)_{i\in I}$ in $\L(E,F)\times\L_{\W,V}(E,F)$ be multi-bounded above and assume that $|I|<\kappa+1$. For $R\in\L(E,F)$, we have $R\in\underset{i\in I}{\minf}\big(T_{i},\L_{W_{i},V}(E,F)\big)$ if and only if there exists $L\in\L(E,F)$ such that $L(\sum_{i\in I}W_{i})\subseteq\D(V)$ and
\[
R(x)=P_{U}\left(\minf\left\{\sum_{i\in I}T_{i}(y_{i}):y_{i}\in W_{i},\sum_{i\in I}y_{i}=x\right\}\right)+L(x)\quad \left(x\in\sum_{i\in I}W_{i}\right).
\]
In particular, if $\sum_{i\in I}W_{i}$ is generating and $V$ is a cone then $\underset{i\in I}{\minf}\big(T_{i},\L_{W_{i},V}(E,F)\big)$ is proper and
\[
\underset{i\in I}{\minf}\big(T_{i},\L_{W_{i},V}(E,F)\big)(x)=\inf\left\{\sum_{i\in I}T_{i}(y_{i}):y_{i}\in W_{i},\sum_{i\in I}y_{i}=x\right\}\quad \left(x\in\sum_{i\in I}W_{i}\right).
\]
\item[(5)] If $\kappa=\aleph_{0}$ then the formulas in (3) and (4) hold for any cardinality of $I$.
\item[(6)] If $W_{i}$ is generating for each $i\in I$ and $V$ is a cone then the multi-supremum in (3) and the multi-infimum in (4) are always proper.
\end{itemize}
\end{theorem}

\begin{proof} 
We prove (1), (2), and (3) simultaneously, from which (4), (5), and (6) will follow, using Proposition~\ref{P:ids}(1) for (4). Let $I$ be a set with $|I|<\kappa+1$, and let $W_{i}\in\W\ (i\in I)$ so that $\big(T_{i},\L_{W_{i},V}(E,F)\big)_{i\in I}$ is a collection of elements in $\L(E,F)\times\L_{\W,V}(E,F)$. Suppose that $\big(T_{i},\L_{W_{i},V}(E,F)\big)_{i\in I}$ is multi-bounded above. For brevity, we use the symbols $\leq$, $\geq$, and $\sim$ to indicate $\leq_{V}$, $\geq_{V}$, and $\sim_{V}$, respectively.

By Remark~\ref{R:setofmsup} and Proposition~\ref{P:wedgeofops}, it suffices to prove that the map $R\colon\sum_{i\in I}W_{i}\rightarrow F$ defined by
\[
R(x)=P_{U}\left(\msup\left\{\sum_{i\in I}T_{i}(y_{i}):y_{i}\in W_{i},\sum_{i\in I}y_{i}=x\right\}\right)
\]
extends linearly to an element of $\underset{i\in I}{\msup}\big(T_{i},\L_{W_{i},V}(E,F)\big)$.

We begin by showing that $R$ is well-defined. Since $\big(T_{i},\L_{W_{i},V}(E,F)\big)_{i\in I}$ is multi-bounded above, there exists $S\in\L(E,F)$ such that $T_{i}(y_{i})\leq  S(y_{i})\ (y_{i}\in W_{i},i\in I)$. Next fixing $x\in\sum_{i\in I}W_{i}$, we have for all $y_{i}\in W_{i}\ (i\in I)$ for which $\sum_{i\in I}y_{i}=x$ that
\[
\sum_{i\in I}T_{i}(y_{i})\leq \sum_{i\in I}S(y_{i})=S(x).
\]
Thus the collection
\[
\left\{\sum_{i\in I}T_{i}(y_{i}):y_{i}\in W_{i},\sum_{i=1}^{n}y_{i}=x\right\}
\]
is multi-bounded above in $F$ with respect to the single wedge $V$. Hence
\[
\msup\left\{\sum_{i\in I}T_{i}(y_{i}):y_{i}\in W_{i},\sum_{i\in I}y_{i}=x\right\}\neq\varnothing.
\]
It follows from Lemma~\ref{L:P}(2) that $P_{U}\big(\msup\left\{\sum_{i\in I}T_{i}(y_{i}):y_{i}\in W_{i},\sum_{i\in I}y_{i}=x\right\}\big)$ is a single element in $F$, and hence $R$ is well-defined.

We next show that $R$ is additive. To this end, let $x_{1},x_{2}\in\sum_{i\in I}W_{i}$. For $x\in\sum_{i\in I}W_{i}$, define
\[
M_{x}=\msup\left\{\sum_{i\in I}T_{i}(y_{i}):y_{i}\in W_{i},\sum_{i\in I}y_{i}=x\right\}.
\]
If $y_{i},w_{i}\in W_{i}\ (i\in I)$, $\sum_{i\in I}y_{i}=x_{1}$, and $\sum_{i\in I}w_{i}=x_{2}$ then $y_{i}+w_{i}\in W_{i}\ (i\in I)$ and $\sum\limits_{i\in I}(y_{i}+w_{i})=x_{1}+x_{2}$. Moreover,
\[
\sum_{i\in I}T_{i}(y_{i})+\sum_{i\in I}T_{i}(w_{i})=\sum_{i\in I}T_{i}(y_{i}+w_{i})\leq M_{x_{1}+x_{2}}.
\]
By taking multi-suprema, we obtain
\[
M_{x_{1}}+M_{x_{2}}\leq M_{x_{1}+x_{2}}.
\]
On the other hand, suppose $y_{j}\in W_{j}\ (j\in I)$ and $x_{1}+x_{2}=\sum_{j\in I}y_{j}$. Since $(E,\W)$ has the $(2,\kappa)$-Riesz Decomposition property, there exists $z_{1j},z_{2j}\in W_{j}\ (j\in I)$ such that
\[
x_{1}=\sum_{j\in I}z_{1j},\quad x_{2}=\sum_{j\in I}z_{2j},\quad \text{and}\quad y_{j}=z_{1j}+z_{2j}\ (j\in I).
\]
Then
\[
\sum_{j\in I}T_{j}(y_{j})=\sum_{j\in I}T_{j}(z_{1j}+z_{2j})=\sum_{j\in I}T_{j}(z_{1j})+\sum_{j\in I}T_{j}(z_{2j})\leq M_{x_{1}}+M_{x_{2}}.
\]
By taking multi-suprema, we get $M_{x_{1}+x_{2}}\leq M_{x_{1}}+M_{x_{2}}$. Thus $M_{x_{1}+x_{2}}\sim M_{x_{1}}+M_{x_{2}}$. Therefore, by Lemma~\ref{L:P}(2),
\[
R(x_{1}+x_{2})=P_{U}(M_{x_{1}+x_{2}})=P_{U}(M_{x_{1}}+M_{x_{2}})=P_{U}(M_{x_{1}})+P_{U}(M_{x_{2}})=R(x_{1})+R(x_{2}).
\]
Hence $R$ is additive on $\sum_{i\in I}W_{i}$. Note that if $\kappa=\aleph_{0}$ then the condition $|I|<\kappa+1$ can be dropped in the argument above. Thus if $\kappa=\aleph_{0}$ then the additivity of $R$ holds regardless of the cardinality of $I$.

The additivity of $R$ implies $R(0)=0$. It now follows from Proposition~\ref{P:ids}(3) that $R$ is positively homogeneous. By Lemma~\ref{L:extn}, we know $R$ extends to an element of $\L(E,F)$, again denoted by $R$.

We next show that $R\in\underset{i\in I}{\msup}\big(T_{i},\L_{W_{i},V}(E,F)\big)$. To this end, fix $i\in I$. For $x\in W_{i}$, we have by Lemma~\ref{L:P}(1) that
\begin{align*}
R(x)\sim\msup\left\{\sum_{i\in I}T_{i}(y_{i}):y_{i}\in W_{i},\sum_{i\in I}y_{i}=x\right\}\geq T_{i}(x).
\end{align*}
Thus $R\geq_{i}T_{i}$ for all $i\in I$. Moreover, if $S\in\L(E,F)$ satisfies
\[
S\left(y_{i}\right)\geq T_{i}(y_{i})\quad (y_{i}\in W_{i}, i\in I)
\]
then for $y_{i}\in W_{i}\ (i\in I)$ with $\sum_{i\in I}y_{i}=x$ we get
\[
S(x)=S\left(\sum_{i\in I}y_{i}\right)=\sum_{i\in I}S\left(y_{i}\right)\geq \sum_{i\in I}T_{i}(y_{i}).
\]
Thus $S(x)\geq\msup\left\{\sum_{i\in I}T_{i}(y_{i}):y_{i}\in W_{i},\sum_{i\in I}y_{i}=x\right\}\sim R(x)$ for all $x\in\sum_{i\in I}W_{i}$. In particular, for each $i\in I$, we have $S(x)\geq R(x)$ for all $x\in W_{i}$. Hence $S\geq_{i}R$ for all $i\in I$. We conclude that $R\in\underset{i\in I}{\msup}\big(T_{i},\L_{W_{i},V}(E,F)\big)$. Noting that the properness condition in (3) follows from Corollary~\ref{C:pmsupofops}, the proof for parts (1), (2), and (3) is complete.
\end{proof}

Let $(E,W)$ and $(F,V)$ be preordered vector spaces. We say that a map $T\colon E\rightarrow F$ is \textit{order bounded} if for every $A\subseteq E$ for which $x\leq A\leq y$ for some $x,y\in E$, there exists $w,z\in F$ such that $w\leq T(A)\leq z$. It is easily checked that the set of all order bounded linear maps $\L_{b}(E,F)$ from $E$ to $F$ forms a preordered vector space. It follows from Proposition~\ref{P:wedgeofops} that if $W$ is generating and $V$ is a cone then $\L_{b}(E,F)$ is an ordered vector space. Under the assumption that $V$ is a cone, it is readily checked that if $I$ is finite and $(T_{i})_{i\in I}$ is a collection of elements of $\L_{b}(E,F)$ then for all $x\in W$,
\[
\left\{\sum_{i\in I}T_{i}(y_{i}):y_{i}\in W,\sum_{i\in I}y_{i}=x\right\}
\]
is bounded above. Moreover, using two consecutive induction arguments, one easily verifies that a preordered vector space with the $(2,2)$-Riesz decomposition property also has the $(\aleph_{0},\aleph_{0})$-Riesz decomposition property. Using these facts in conjunction with Theorem~\ref{T:main}, one readily proves our next corollary. Note that $(F,V)$ in the Corollary~\ref{C:classic} is not necessarily a vector lattice since sets of two elements in a general ordered vector space need not be bounded above.

\begin{corollary}\label{C:classic}
Suppose $(E,W)$ is a preordered vector space with the $(2,2)$-Riesz decomposition property and that $W$ is generating. Let $(F,V)$ be an ordered vector space such that for every nonempty $A\subseteq F$ that is bounded above, $\sup A$ exists in $F$ (that is, $(F,\{V\})$ is a Dedekind complete multi-lattice). Then $\L_{b}(E,F)$ is a Dedekind complete vector lattice. For a finite or infinite and bounded-above collection $(T_{i})_{i\in I}$ in $\L_{b}(E,F)$ and $x\in W$,
\[
\underset{i\in I}{\sup}(T_{i})(x)=\sup\left\{\sum_{i\in I}T_{i}(y_{i}):y_{i}\in W,\sum_{i\in I}y_{i}=x\right\}.
\]
\end{corollary}

\begin{remark}\label{R:classical}
Reducing Corollary~\ref{C:classic} to the classical case, the above formula may seem new, since most authors seem to only consider the Riesz Kantorovich formulas for a finite or bounded above upwards directed infinite collection of operators. But in fact, the somewhat hidden \cite[Lemma~36.3]{LilZan} provides similar Riesz-Kantorovich formulas for the supremum of any finite or bounded above infinite collection of operators. The previous corollary shows that \cite[Lemma~36.3]{LilZan} also holds in this more general setting.
\end{remark}

The special case in Theorem~\ref{T:main} where $(F,\{V\})=(\R,\{\R^{+}\})$, as in the classical theory, is of particular interest. For a multi-wedged space $(E,\W)$, let $E'$ denote the algebraic dual of $E$, and let $\W'$ be the collection $\{W':W\in\W\}$, where $W'\subseteq E'$ is the dual wedge of $W\in\W$. We define the \textit{dual} of $(E,\W)$ to be the multi-wedged space $(E',\W')$. 

\begin{corollary}\label{C:main}
	Let $(E,\W)$ be a multi-wedged space with the $(2,\kappa)$-Riesz decomposition property, where $\kappa\in\N\cup\{\aleph_{0}\}$.
	\begin{itemize}
		\item[(1)] $(E',\W')$ is a $\kappa$-multi-lattice.
		\item[(2)] If $\kappa=\aleph_{0}$ then $(E',\W')$ is Dedekind complete.
		\item[(3)] Let $(\phi_{i},W_{i}')_{i\in I}$ be multi-bounded above in $E'\times\W'$ with $|I|<\kappa+1$. For $\psi\in E'$, we have $\psi\in\underset{i\in I}{\msup}(\phi_{i},W_{i}')$ if and only if there exists $\xi\in E'$ such that $\xi(\sum_{i\in I}W_{i})=\{0\}$ and
		\[
		\psi(x)=\sup\left\{\sum_{i\in I}\phi_{i}(y_{i}):y_{i}\in W_{i},\sum_{i\in I}y_{i}=x\right\}+\xi(x)\quad \left(x\in\sum_{i\in I}W_{i}\right).
		\]
		In particular, if $\sum_{i\in I}W_{i}$ is generating then $\underset{i\in I}{\msup}\big(\phi_{i},W_{i}'\big)$ is proper and
		\[
		\underset{i\in I}{\msup}\big(\phi_{i},W_{i}'\big)(x)=\sup\left\{\sum_{i\in I}\phi_{i}(y_{i}):y_{i}\in W_{i},\sum_{i\in I}y_{i}=x\right\}\quad \left(x\in\sum_{i\in I}W_{i}\right).
		\]
		\item[(4)] Let $(\phi_{i},W_{i}')_{i\in I}$ be multi-bounded below in $E'\times\W'$ with $|I|<\kappa+1$. For $\psi\in E'$, we have $\psi\in\underset{i\in I}{\minf}(\phi_{i},W_{i}')$ if and only if there exists $\xi\in E'$ such that $\xi(\sum_{i\in I}W_{i})=\{0\}$ and
		\[
		\psi(x)=\inf\left\{\sum_{i\in I}\phi_{i}(y_{i}):y_{i}\in W_{i},\sum_{i\in I}y_{i}=x\right\}+\xi(x)\quad \left(x\in\sum_{i\in I}W_{i}\right).
		\]
		In particular, if $\sum_{i\in I}W_{i}$ is generating then $\underset{i\in I}{\minf}\big(\phi_{i},W_{i}'\big)$ is proper and
		\[
		\underset{i\in I}{\minf}\big(\phi_{i},W_{i}'\big)(x)=\inf\left\{\sum_{i\in I}\phi_{i}(y_{i}):y_{i}\in W_{i},\sum_{i\in I}y_{i}=x\right\}\quad \left(x\in\sum_{i\in I}W_{i}\right).
		\]
		\item[(5)] If $\kappa=\aleph_{0}$ then the formulas in (3) and (4) hold for any cardinality of $I$.
		\item[(6)] If $W_{i}$ is generating for each $i\in I$ then the multi-supremum in (3) and the multi-infimum in (4) are always proper.
	\end{itemize}
\end{corollary}

We conclude with an example, which relies on the corollary above and illustrates several fundamental concepts from this paper.

\begin{example}\label{E:F(S)}
Let $S$ be a nonempty set, and consider the vector space $F(S)$ of real-valued functions on $S$. For $s \in S$, let $W_s = \{ f \in F(S): f(s) \geq 0 \}$,
and let $\mathcal{W} = \{ W_s: s \in S\}$. Then $(F(S), \W)$ is a multi-wedged space. If $(f_i, W_{s_i})_{i\in I}$ is multi-bounded above then for any multi-upper bound $g$, we have $g(s_{i})\geq f_i(s_{i})$ for all $i\in I$. It follows that
$$ \msup_{i\in I}(f_i, W_{s_i}) = \big\{f \in F(S): f(s_{i}) = \sup \{ f_j(s_{i}): s_j = s_{i}\}\ \text{for all}\ i\in I \big\}. $$
Hence $(F(S), \W)$ is a Dedekind complete multi-lattice.

Consider the subspace $F_0(S)$ of all functions $f \in F(S)$ with finite support and the corresponding restrictions of the wedges $\W = \{W_s: s \in S\}$, also denoted by $\W$. Then $(F_0(S), \W)$ is a multi-wedged space as well, which inherits the $\aleph_0$-multi-lattice structure from $(F(S), \W)$. However, if $S$ is infinite then $F_0(S)$ is not Dedekind complete. To illustrate, define $e_{s}(t)=\delta_{st}\ (s,t\in S)$, where $\delta_{st}$ is the Kronecker delta function. Note that $(-e_s, W_s)_{s\in S}$ is multi-bounded above (by $0$). If $f \in F_0(S)$ is a multi-upper bound of $(-e_s, W_s)_{s\in S}$ then for some $s_{0}$ not in the support of $f$, define $g\in F_{0}(S)$ by
\[
g(s)=\begin{cases} f(s) &\mbox{if } s\neq s_{0} \\
-1 & \mbox{if } s=s_{0} \end{cases}.
\]
Then $g$ is another multi-upper bound of $(-e_s, W_s)_{s\in S}$, but $f\nleq_{s_{0}}g$.

We next show that $F(S)$ and $F_0(S)$ have the $(\aleph_0,\aleph_0)$-Riesz decomposition property. Let $y_1 \in W_{s_1}, \ldots, y_m \in W_{s_m}$ and $x_1, \ldots, x_n \in \sum_{j=1}^m W_{s_j}$ be such that $\sum_{i=1}^n x_i = \sum_{j=1}^m y_j$. We seek functions $z_{ij}\in W_{s_{j}}\ (1\leq i\leq n, 1\leq j\leq m)$ for which
$$ x_i = \sum_{j=1}^m z_{ij}\quad (1\leq i\leq n) \quad \mbox{and} \quad y_j = \sum_{i=1}^n z_{ij}\quad (1\leq j\leq m). $$
We may assume that $n,m \geq 2$, otherwise the result is trivial. First suppose that all wedges $W_{s_j}$ are the same, so $s_1 = \cdots = s_m$. In this case one can easily find $z_{ij}(s)$ satisfying the requirements for any $s \not= s_1$. For $s=s_1$, we simply use the Riesz decomposition property of $\R$ to find the required $z_{ij}(s_1)$.

Therefore, we assume without loss of generality that $W_{s_1} \not= W_{s_2}$. In this case, we claim that it suffices to find elements $z_{ij}\ (1 \leq i \leq n-1, 1 \leq j \leq m)$ with the following properties.
\begin{itemize}
	\item[(1)] For $1 \leq i \leq n-1$ and $1 \leq j \leq m$, we have $z_{ij}(s_j) = 0$.
	\item[(2)] For $1 \leq i \leq n-1$, we have $\sum_{j=1}^m z_{ij} = x_i$.
\end{itemize}
Indeed, if (1) and (2) are satisfied then for $1 \leq j \leq m$ and $z_{nj} = y_j - \sum_{i=1}^{n-1} z_{ij} \in W_{s_j}$,
$$ \sum_{j=1}^m z_{nj} = \sum_{j=1}^m (y_j - \sum_{i=1}^{n-1} z_{ij}) = \sum_{j=1}^m y_j - \sum_{i=1}^{n-1} x_i = x_n $$
and
\[
\sum_{i=1}^{n}z_{ij}=y_{j}.
\]
For $1 \leq i \leq n-1$, define
\[ \begin{array}{ccc}
z_{i1}(s)=\begin{cases} 0 &\mbox{if } s=s_{1} \\
x_i(s) & \mbox{if } s \not= s_1 \end{cases}\ & \text{and}\ & z_{i2}(s)=\begin{cases} x_{i}(s) &\mbox{if } s=s_{1} \\
0 & \mbox{if } s \not= s_1 \end{cases}.
\end{array} \]
Set $z_{ij} = 0$ for all $1 \leq i \leq n-1$ and $3 \leq j \leq m$. Now it is easily verified that properties (1) and (2) are satisfied.

For $s \in S$, define the wedge $V_s = \{\lambda e_s: \lambda \geq 0\}$ in $F(S)$. Under the canonical duality $\langle f,g\rangle=\sum_{s\in S}f(s)g(s)$, the algebraic dual of $F_{0}(S)$ is $F(S)$, and one easily verifies that the dual wedge of $W_s$ is $V_s\ (s\in S)$. By Corollary~\ref{C:main}, we see that $(F(S),\{V_{s}:s\in S\})$ is a Dedekind complete multi-lattice in which every multi-supremum is proper, since all wedges $W_s$ are generating.
\end{example}

\bibliography{mybib}

\providecommand{\bysame}{\leavevmode\hbox to3em{\hrulefill}\thinspace}
\providecommand{\MR}{\relax\ifhmode\unskip\space\fi MR }
\providecommand{\MRhref}[2]{%
  \href{http://www.ams.org/mathscinet-getitem?mr=#1}{#2}
}
\providecommand{\href}[2]{#2}
\begin{thebibliography}{1}

\bibitem{ATY2}
C.~D. Aliprantis, R.~Tourky, and N.~C. Yannelis, \emph{The
  {R}iesz-{K}antorovich formula and general equilibrium theory}, J. Math.
  Econom. \textbf{34} (2000), no.~1, 55--76.

\bibitem{ATY}
\bysame, \emph{A theory of value with non-linear prices. {E}quilibrium analysis
  beyond vector lattices}, J. Econom. Theory \textbf{100} (2001), no.~1,
  22--72.

\bibitem{AB}
C.D. Aliprantis and O.~Burkinshaw, \emph{Positive {O}perators}, Academic Press,
  Orlando, 1985.

\bibitem{AT}
C.D. Aliprantis and R.~Tourky, \emph{Cones and {D}uality}, Graduate Studies in
  Mathematics, vol.~84, American Mathematical Society, Providence, RI, 2007.

\bibitem{Ando}
T.~And{\^o}, \emph{On fundamental properties of a {B}anach space with a cone},
  Pacific J. Math. \textbf{12} (1962), 1163--1169.

\bibitem{JeuMess}
M.~de~Jeu and M.~Messerschmidt, \emph{A strong open mapping theorem for
  surjections from cones onto {B}anach spaces}, Adv. Math. \textbf{259} (2014),
  43--66.

\bibitem{Mess}
M.~Messerschmidt, \emph{Geometric duality theory of cones in dual pairs of
  vector spaces}, J. Funct. Anal. \textbf{269} (2015), no.~7.

\bibitem{LilZan}
A.C. Zaanen, \emph{Introduction to operator theory in {R}iesz spaces},
  Springer-Verlag, Berlin, 1997.

\end{thebibliography}
\bibliographystyle{amsplain}

\end{document}